\documentclass[a4paper,10pt]{article}
\usepackage{amsthm,amsmath,amssymb,graphicx}
\usepackage{hyperref}
\hypersetup{
    colorlinks=true,       
    linkcolor=blue,        
    citecolor=red,         
    filecolor=magenta,     
    urlcolor=cyan,         
    linktocpage=true
}

\renewcommand{\title}[1]{\vspace{\fill}
  \eject\addtolength{\baselineskip}{4pt}
  {\bfseries\LARGE #1}\\[3mm]\addtolength{\baselineskip}{-4pt}}
\renewcommand{\author}[3]{\parbox[t]{75mm}
  {\begin{center}{\scshape #1}\\[3mm] #2\\
      {\ttfamily #3} \end{center}}}

\newtheorem{thm}{\bfseries Theorem}
\newtheorem{cor}[thm]{\bfseries Corollary}

\usepackage{algorithm}
\usepackage{algpseudocode}
\usepackage{tikz}
\usepackage{tkz-berge}
\usepackage{mathtools}
\usepackage{subcaption}
\usepackage{bbold}
\usepackage{xfrac}
\usepackage{graphicx}

\definecolor{officegreen}{rgb}{0.0, 0.5, 0.0}

\pgfdeclarelayer{background}
\pgfdeclarelayer{foreground}
\pgfsetlayers{background,main,foreground}

\tikzset{
    my box/.style = {
        , line cap = round
        , line join = round
    }
}

\usetikzlibrary{snakes}
\tikzset{wavy/.style={decorate,decoration={snake,amplitude=.4mm,segment length=2mm,post length=0mm,pre length=0mm},line width=.5}}

\SetVertexNormal[FillColor=gray!25]

\makeatletter
\renewcommand*{\write@math}[3]{%
            \pgfmathtruncatemacro{\printindex}{#3+1}
            \Vertex[x = #1,y = #2,%
                    L = \cmdGR@cl@prefix\grMathSep{\printindex}]{\cmdGR@cl@prefix#3}}

\def\1{\mathbb 1}

\begin{document}

\begin{center}

  \title{The Ramsey number of a long cycle and complete graphs}
  \author{
    P\'eter Madarasi
  }{
    Department of Operations Research \\
    E\"otv\"os Lor\'and University\\
    Budapest, Hungary
  }{
    madarasi@cs.elte.hu
  }


\end{center}


\begin{quote}
  {\bfseries Abstract:}
  In this paper, we prove that the multicolored Ramsey number $R(G_1,\dots,G_n,K_{n_1},\dots,K_{n_r})$ is at least $(\gamma-1)(\kappa-1)+1$ for arbitrary connected graphs $G_1,\dots,G_n$ and $n_1,\dots,n_r\in\mathbb{N}$, where $\gamma=R(G_1,\dots,G_n)$ and $\kappa=R(K_{n_1},\dots,K_{n_r})$.
Erd\H os at al. conjectured that $R(C_n,K_l)=(n-1)(l-1)+1$ for every $n\geq l\geq 3$ except for $n=l=3$ \cite{erdosConj}. Nikiforov proved this conjecture for $n\geq 4l+2$.
Using the above bound, we derive the following generalization of this result. $R(C_n,K_{n_1},\dots,K_{n_r})=(n-1)(\kappa-1)+1$, where $\kappa=R(K_{n_1},\dots,K_{n_r})$ and $n\geq 4\kappa+2$.
\end{quote}

\begin{quote}
  \textbf{Keywords:} Multicolored Ramsey number, Ramsey number of a long cycle and complete graphs, Ramsey number of connected graphs
\end{quote}
\vspace{5mm}

\section{Introduction}\label{sec:intro}
The Ramsey number $R(G_1, G_2,\dots, G_n)$ is the smallest positive integer $p$ such that if $(E_1, E_2, ,\dots, E_n)$ is an arbitrary partition of the edges of the complete graph $K_p$, then the edge-induced subgraph $\langle E_i\rangle$ contains a graph isomorphic to $G_i$ for some $i$. The classes $E_1, E_2,\dots, E_n$ can be visualized as color classes and the partition $(E_1, E_2,\dots, E_n)$ is thought of as an edge coloring using $n$ colors. Erd\H os at al. conjectured that $R(C_n,K_l)=(n-1)(l-1)+1$ for every $n\geq l\geq 3$ except for $n=l=3$ \cite{erdosConj}. Nikiforov proved this conjecture for $n\geq 4l+2$.
In this paper, we give a lower bound for $R(G_1,\dots,G_n,K_{n_1},\dots,K_{n_r})$ in terms of $R(G_1,\dots,G_n)$ and $R(K_{n_1},\dots,K_{n_r})$, where $G_1,\dots,G_n$ are connected graphs. Using this bound, we prove a generalization of Nikiforov's theorem~\cite{Nikiforov}.

\section{New results}
First, consider the following lower bound for $R(G_1,\dots,G_n,K_{n_1},\dots,K_{n_r})$.
\begin{thm}\label{thm:bound}
For any connected graphs $G_1,\dots,G_n$,
\begin{equation}
  R(G_1,\dots,G_n,K_{n_1},\dots,K_{n_r})\geq(\gamma-1)(\kappa-1)+1,
  \end{equation}
where $\gamma=R(G_1,\dots,G_n)$ and $\kappa=R(K_{n_1},\dots,K_{n_r})$.
\end{thm}
\begin{proof}
Consider an $r$-coloring of the edges of $K_{\kappa-1}$ with colors $c_1,\dots,c_r$ such that the graph induced by color class $c_i$ does not contain $K_{n_i}$ as a subgraph for all $i\in\{1,\dots,r\}$. Replace each node of $K_{\kappa-1}$ by the complete graph $K_{\gamma-1}$ whose edges are colored by colors $d_1,\dots,d_n$ such that the $i^\emph{th}$ color class does not contain $G_i$ as subgraph for all $i\in\{1,\dots,n\}$, and expand each original edge $uv$ of $K_{\kappa-1}$ to a complete bipartite graph between the two $K_{\gamma-1}$ graphs that replace nodes $u$ and $v$. Let the new edges added in the latter step inherit the color of $uv$. This way, $K_{(\kappa-1)(\gamma-1)}$ has been colored using colors $c_1,\dots,c_r,d_1,\dots,d_n$. For each $i\in\{1,\dots,n\}$, color class $d_i$ does not contain $G_i$ as a subgraph, because $G_i$ is connected and each connected component of color $d_i$ in $K_{(\kappa-1)(\gamma-1)}$ is induced by a copy of $K_{\gamma-1}$. By contradiction, if the edges with color $c_i$ contained $K_{n_i}$ as a subgraph, $K_{n_i}$ contained at most one node of each copy of $K_{\gamma-1}$ (as they have no edges with color $c_i$), meaning that it is a subgraph of $K_{\kappa-1}$ with color $c_i$ -- contradicting the definition of $\kappa$.
\end{proof}

In what follows, the bound given by Theorem~\ref{thm:bound} is shown to be tight in a special case, which will be used to derive the exact value of $R(C_n,K_{n_1},\dots,K_{n_r})$ in terms of $R(K_{n_1},\dots,K_{n_r})$ if cycle $C_n$ is sufficiently long.

\begin{thm}\label{thm:tight}
  For any connected graphs $G_1,\dots,G_n$, if $R(G_1,\dots,G_n,K_\kappa)=(\gamma-1)(\kappa-1)+1$, then
  \begin{equation}
    R(G_1,\dots,G_n,K_{n_1},\dots,K_{n_r})=(\gamma-1)(\kappa-1)+1,    
  \end{equation}
where $\gamma=R(G_1,\dots,G_n)$ and $\kappa=R(K_{n_1},\dots,K_{n_r})$.
\end{thm}
\begin{proof}
Consider an arbitrary $(m+n)$-edge coloring of $K_{(\gamma-1)(\kappa-1)+1}$ using co\-lors $c_1,\dots,c_r,d_1,\dots,d_n$. In what follows, we argue that either the graph induced by color class $c_i$ contains $K_{n_i}$ as a subgraph for some $i\in\{1,\dots,r\}$ or color class $d_i$ contains $G_i$ as subgraph for some $i\in\{1,\dots,n\}$.
Let us recolor color classes $c_1,\dots,c_r$ to a single new color $c$ and keep the colors of the rest of edges. This way one gets a coloring of $K_{(\gamma-1)(\kappa-1)+1}$ with colors $d_1,\dots,d_n$ and $c$. Since $R(G_1,\dots,G_n,K_\kappa)=(\gamma-1)(\kappa-1)+1$, the new coloring contains $G_i$ in color $d_i$ for some $i\in\{1,\dots,n\}$ or it includes $K_\kappa$ in color $c$. In the former case, we have that the original coloring (the one that uses colors $c_1,\dots,c_r,d_1,\dots,d_n$) contains $G_i$ in color $d_i$, hence we are done. Otherwise, the edges of $K_\kappa$ with color $c$ (in the modified coloring) contain $K_\kappa$ as subgraph. Recovering the original colors ($c_1,\dots,c_r$) of each edge, one gets by the definition of $\kappa$ that the graph induced by color class $c_i$ contains $K_{n_i}$ as a subgraph for at least one $i\in\{1,\dots,r\}$. Hence for any $(m+n)$-edge coloring of $K_{(\gamma-1)(\kappa-1)+1}$, the graph induced by color class $c_i$ contains $K_{n_i}$ as a subgraph for some $i\in\{1,\dots,r\}$ or color class $d_i$ contains $G_i$ as subgraph for some $i\in\{1,\dots,n\}$, which completes the proof.
\end{proof}

In what follows, an interesting corollary of Theorem~\ref{thm:tight} is presented. Erd\H os at al. conjectured that $R(C_n,K_l)=(n-1)(l-1)+1$ for every $n\geq l\geq 3$ except for $n=l=3$~\cite{erdosConj}. Nikiforov proved this conjecture for $n\geq 4l+2$, see~\cite{Nikiforov}. By Theorem~\ref{thm:tight}, one gains the following generalization of this result.


\begin{cor}\label{cor:cycle}
$R(C_n,K_{n_1},\dots,K_{n_r})=(n-1)(\kappa-1)+1$, where\\$\kappa=R(K_{n_1},\dots,K_{n_r})$ and $n\geq 4\kappa+2$.
\end{cor}
\begin{proof}
Let $\gamma=R(C_n)=n$ and $\kappa=R(K_{n_1},\dots,K_{n_r})$. By Nikiforov's theorem, one gets that $R(C_n,K_\kappa)=(n-1)(\kappa-1)+1=(\gamma-1)(\kappa-1)+1$. That is, the conditions of Theorem~\ref{thm:tight} are fulfilled, hence $R(C_n,K_{n_1},\dots,K_{n_r})=(\gamma-1)(\kappa-1)+1$, which completes the proof.
\end{proof}

Figure~\ref{fig:example} presents the construction for $R(C_{22},K_3,K_3)$.
In fact, the following more general result holds for $\kappa\in\{3,\dots,7\}$.

\begin{cor}\label{cor:cycle2}
$R(C_n,K_{n_1},\dots,K_{n_r})=(n-1)(\kappa-1)+1$, where\\$\kappa=R(K_{n_1},\dots,K_{n_r})\in\{3,\dots,7\}$ and $n\geq \kappa$.
\end{cor}
\begin{proof}
  Similarly to the proof of Corollary~\ref{cor:cycle}, one has to show that $R(C_n,K_\kappa)=(n-1)(\kappa-1)+1$ holds. For $\kappa=3,4,5,6,7$, this was shown in \cite{Rosta73}, \cite{Yang99}, \cite{Bollobas00}, \cite{Schiermeyer03} and \cite{CHEN20081337}, respectively.
\end{proof}

  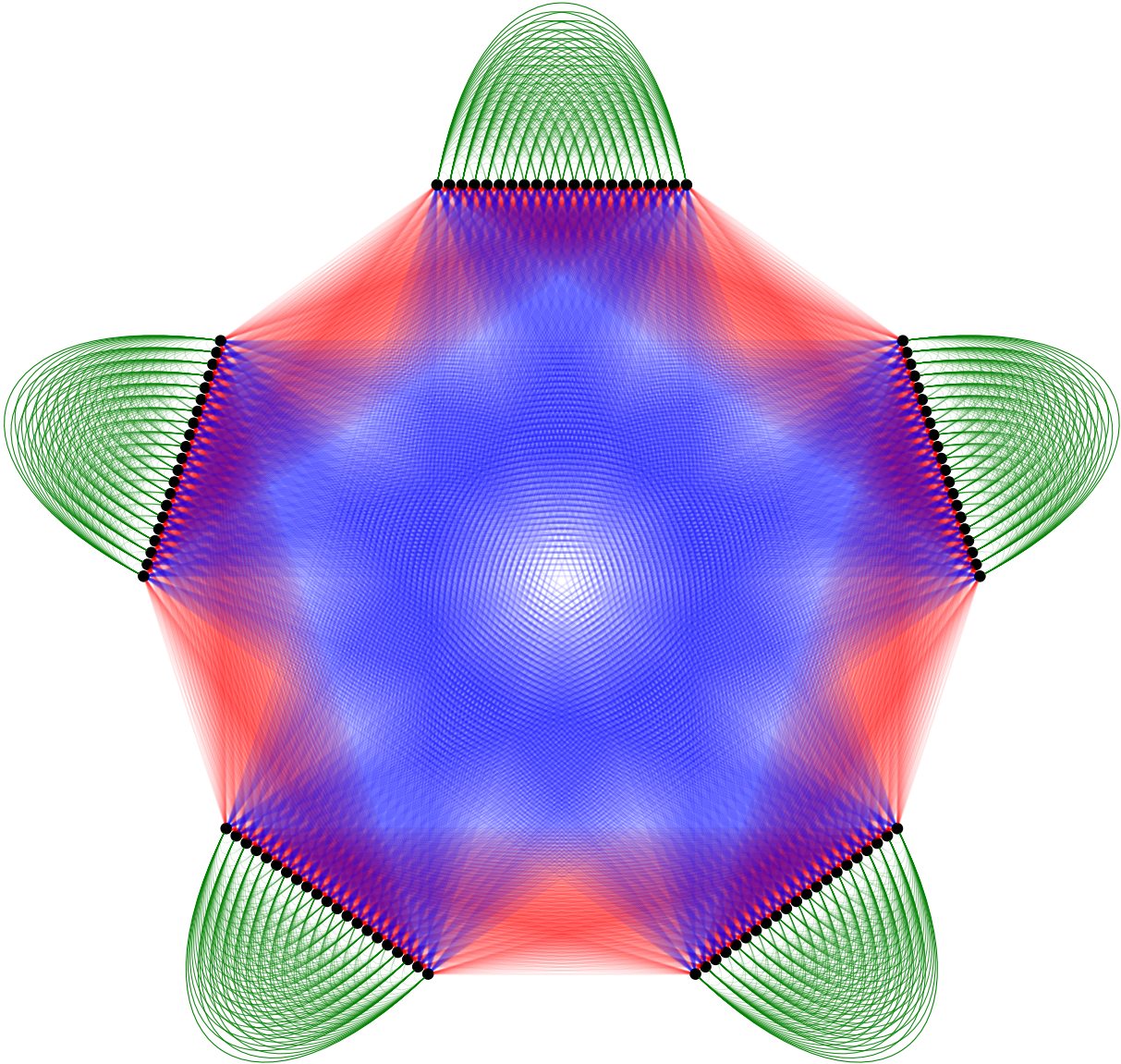
\begin{figure}[H]
    \centering
    \makebox[\textwidth][c]{

  \begin{tikzpicture}[scale=.71,rotate=180]
    \centering
    \SetVertexNoLabel
  \begin{scope}[rotate=0]
    \tikzset{VertexStyle/.append style={color=black,minimum size=1.5pt, inner sep=1.5pt}}
    \grEmptyPath[x=-20*.25/2,y=-8,RA=.25,prefix=a]{21}{0}
     \foreach \i in {0,...,20}{
    \foreach \j in {\i,...,20}{
          \draw (a\i) edge[color=officegreen,bend right=80,looseness=2.5,opacity=(1+\j-\i)/21] node[left] {} (a\j);
    }
  }
  \end{scope}
  \begin{scope}[rotate=360/5]
    \tikzset{VertexStyle/.append style={color=black,minimum size=1.5pt, inner sep=1.5pt}}
    \grEmptyPath[x=-20*.25/2,y=-8,RA=.25,prefix=b]{21}{0}
    \foreach \i in {0,...,20}{
        \foreach \j in {\i,...,20}{
              \draw (b\i) edge[color=officegreen,bend right=80,looseness=2.5,opacity=(1+\j-\i)/21] node[left] {} (b\j);
    }
  }
  \end{scope}
  \begin{scope}[rotate=2*360/5]
    \tikzset{VertexStyle/.append style={color=black,minimum size=1.5pt, inner sep=1.5pt}}
    \grEmptyPath[x=-20*.25/2,y=-8,RA=.25,prefix=c]{21}{0}
    \foreach \i in {0,...,20}{
        \foreach \j in {\i,...,20}{
              \draw (c\i) edge[color=officegreen,bend right=80,looseness=2.5,opacity=(1+\j-\i)/21] node[left] {} (c\j);
    }
  }
  \end{scope}
  \begin{scope}[rotate=3*360/5]
    \tikzset{VertexStyle/.append style={color=black,minimum size=1.5pt, inner sep=1.5pt}}
    \grEmptyPath[x=-20*.25/2,y=-8,RA=.25,prefix=d]{21}{0}
    \foreach \i in {0,...,20}{
        \foreach \j in {\i,...,20}{
              \draw (d\i) edge[color=officegreen,bend right=80,looseness=2.5,opacity=(1+\j-\i)/21] node[left] {} (d\j);
    }
  }
  \end{scope}
  \begin{scope}[rotate=4*360/5]
    \tikzset{VertexStyle/.append style={color=black,minimum size=1.5pt, inner sep=1.5pt}}
    \grEmptyPath[x=-20*.25/2,y=-8,RA=.25,prefix=e]{21}{0}
       \foreach \i in {0,...,20}{
        \foreach \j in {\i,...,20}{
              \draw (e\i) edge[color=officegreen,bend right=80,looseness=2.5,opacity=(1+\j-\i)/21] node[left] {} (e\j);
    }
  }
  \end{scope}

  \foreach \i in {0,...,20}{
    \foreach \j in {0,...,20}{
      \tikzstyle{EdgeStyle} = [color = red,opacity=.1]
      \Edges(a\i, b\j)
    }
  }
  \foreach \i in {0,...,20}{
    \foreach \j in {0,...,20}{
      \tikzstyle{EdgeStyle} = [color = red,opacity=.1]
      \Edges(b\i, c\j)
    }
  }
\foreach \i in {0,...,20}{
    \foreach \j in {0,...,20}{
      \tikzstyle{EdgeStyle} = [color = red,opacity=.1]
      \Edges(c\i, d\j)
    }
  }
  \foreach \i in {0,...,20}{
    \foreach \j in {0,...,20}{
      \tikzstyle{EdgeStyle} = [color = red,opacity=.1]
      \Edges(d\i, e\j)
    }
  }
\foreach \i in {0,...,20}{
    \foreach \j in {0,...,20}{
      \tikzstyle{EdgeStyle} = [color = red,opacity=.1]
      \Edges(e\i, a\j)
    }
  }

  \foreach \i in {0,...,20}{
    \foreach \j in {0,...,20}{
      \tikzstyle{EdgeStyle} = [color = blue,opacity=.1]
      \Edges(a\i, c\j)
    }
  }
  \foreach \i in {0,...,20}{
    \foreach \j in {0,...,20}{
      \tikzstyle{EdgeStyle} = [color = blue,opacity=.1]
      \Edges(c\i, e\j)
    }
  }
    \foreach \i in {0,...,20}{
    \foreach \j in {0,...,20}{
      \tikzstyle{EdgeStyle} = [color = blue,opacity=.1]
      \Edges(e\i, b\j)
    }
  }
    \foreach \i in {0,...,20}{
    \foreach \j in {0,...,20}{
      \tikzstyle{EdgeStyle} = [color = blue,opacity=.1]
      \Edges(b\i, d\j)
    }
  }
  \foreach \i in {0,...,20}{
    \foreach \j in {0,...,20}{
      \tikzstyle{EdgeStyle} = [color = blue,opacity=.1]
      \Edges(d\i, a\j)
    }
  }

\end{tikzpicture}
}
\vspace*{-15mm}
\caption{An example for Corollary~\ref{cor:cycle}. The largest graph and its edge coloring for which the green, red and blue edges do not contain $C_{22},K_3$ and $K_3$, respectively. The number of nodes is $R(C_{22},K_3,K_3)-1=(22-1)(6-1)=105$. The green edges form five disjoint complete graphs on 21 nodes. The red edges are the union of five complete bipartite graphs between the green complete graphs next to each other, and the blue edges are the union of five complete bipartite graph between the non-adjacent green complete graphs.}
\label{fig:example}
\end{figure}

\end{document}